\pdfoutput=1
\documentclass[11pt]{article}
\usepackage{amsmath, amssymb, amsthm}
\usepackage[a4paper,top=2.5cm,bottom=2.5cm,left=2.8cm,right=2.8cm]{geometry}
\theoremstyle{plain}
\newtheorem{theorem}{Theorem}
\newtheorem{lemma}{Lemma}
\newtheorem{corollary}{Corollary}
\theoremstyle{definition}
\newtheorem{definition}{Definition}
\theoremstyle{remark}
\newtheorem{remark}{Remark}
\newtheorem{problem}{Open Problem}

\title{Circular $s$-choice parking functions:\\ an exact closed formula via rotational symmetry}
\author{Asma Recioui \quad Hac\`ene Belbachir\thanks{Corresponding author: \texttt{hacenebelbachir@gmail.com}, \texttt{hbelbachir@usthb.dz}. ORCID: 0000-0001-8540-3033.} \quad Abdelhakim AitZai\\[2mm]
{\small RECITS Laboratory, USTHB, Algiers, Algeria}\\
{\footnotesize \texttt{reciouiasma@gmail.com}, \texttt{hacenebelbachir@gmail.com}, \texttt{h.aitzai@gmail.com}}}
\date{}

\begin{document}
\maketitle

\begin{abstract}
We study a circular variant of the $s$-choice parking model: $n$ cars park on $m=n+1$ spots arranged on a circle, each car carrying an anchor and $s-1$ clockwise increments at least $d$ apart with return gap at least $d$; a car tries its choices in order and then sweeps clockwise from its last choice. On the circle every car parks and exactly one spot remains empty. Exploiting rotational symmetry in the spirit of Pollak's proof of the count $(n+1)^{n-1}$, we prove that the empty spot is exactly equidistributed, which yields the closed formula $m^{n-1}\binom{m-sd+s-1}{s-1}^{n}$ for the number of preferences leaving any prescribed spot empty. This appears to be the first closed product formula in the multi-choice parking landscape. We further prove a refinement: within every class of preferences with prescribed increments, the empty spot is still exactly equidistributed, which explains the product structure of the formula and, for $s=2$, $d=1$, the appearance of the classical count $(n+1)^{n-1}$ as a factor of $(n+1)^{n-1}n^{n}$. The admissible tuples are enumerated through Kaplansky's lemma on circular selections, and all results are verified by exhaustive computer enumeration.
\end{abstract}


\section{Introduction}\label{sec:intro}

Parking functions are among the most studied objects of enumerative combinatorics, yet closed product formulas for their multi-choice generalizations are essentially absent: the known enumerations are sums over permutations, with no product form. The purpose of this note is to exhibit a natural multi-choice model that does admit an exact closed product formula, and to explain the product structure by a symmetry that survives conditioning. The model is circular, and the method is Pollak's.

Parking functions were introduced by Konheim and Weiss \cite{KonheimWeiss}: $n$ cars enter a one-way street with spots $1,\dots,n$; car $c_i$ drives to its preferred spot and, if it is occupied, proceeds forward to the first free spot; the preference vectors for which all cars park number $(n+1)^{n-1}$. Pollak's celebrated proof --- transferring the process to a circle of $n+1$ spots, where every preference parks all cars and leaves exactly one spot empty, and then invoking rotational symmetry --- is the paradigm we generalize. See Yan's chapter in \cite{Handbook} for a survey.

Multi-choice variants --- each car carrying $s$ preferred spots with consecutive gaps at least $d$, tried in order before the forward search --- can be enumerated on the linear street by summations over the $n!$ permutations, but those summations are not closed forms, and demonstrably cannot be: exhaustive computation of the linear two-choice counts ($1$, $24$, $990$, $65\,940$, $6\,545\,700,\dots$ for $d=1$) reveals sporadic large prime factors such as $157$, $1039$ and $10\,937$, ruling out any product formula. This is precisely what makes the circular model worth isolating. 

Our contribution is threefold. First (Theorem \ref{thm:closed}), the circular $s$-choice model admits the exact closed formula $m^{n-1}\binom{m-sd+s-1}{s-1}^{n}$, obtained by the same rotational symmetry as Pollak's argument; to our knowledge it is the first closed product formula in the multi-choice parking landscape. Second (Theorem \ref{thm:refined}), the symmetry localizes: it holds within each class of preferences with prescribed increments, which explains \emph{why} the count is a product and, for $s=2$ and $d=1$, why the classical number $(n+1)^{n-1}$ appears as a factor of $(n+1)^{n-1}n^{n}$. Third, the admissible tuples are identified with circular selections counted by Kaplansky's lemma, tying the model to the theory of independent sets in powers of cycles. The refinement also opens the systematic study of statistics on the circle, which we leave as a direction for future work.

\section{The circular model}\label{sec:model}

Let $m\geq 2$ and identify the spots with $\mathbb{Z}_m=\{1,\dots,m\}$, arranged clockwise on a circle. For $x\in\mathbb{Z}_m$ and $k\geq 0$ we write $x\oplus k$ for clockwise displacement by $k$ modulo $m$.

\begin{definition}[Admissible tuples]\label{def:ctuples}
Fix $s\geq 1$ and $d\geq 1$. An admissible $s$-tuple is
$$(a_1,\dots,a_s)=\bigl(a,\ a\oplus k_1,\ \dots,\ a\oplus(k_1+\cdots+k_{s-1})\bigr)$$
with $a\in\mathbb{Z}_m$, increments $k_i\geq d$, and $k_1+\cdots+k_{s-1}\leq m-d$, so the return gap from $a_s$ to $a_1$ is at least $d$. Their number is
\begin{equation}\label{eq:tuples}
m\cdot\#\{(k_1,\dots,k_{s-1}):k_i\geq d,\ \textstyle\sum k_i\leq m-d\}=m\binom{m-sd+s-1}{s-1},
\end{equation}
vanishing precisely when $m<sd$.
\end{definition}

\begin{definition}[Circular parking]\label{def:cpark}
Let $n:=m-1$ cars $c_1,\dots,c_n$ enter in order, car $c_j$ carrying an admissible tuple. Car $c_j$ tries its choices in order and parks in the first free one; if all are occupied, it sweeps clockwise from its last choice. A \emph{circular preference} is a choice of admissible tuples for all cars; the set $\mathcal{C}_m^{s,d}$ has cardinality $[m\binom{m-sd+s-1}{s-1}]^n$ by \eqref{eq:tuples}.
\end{definition}

Since $n=m-1<m$, the sweep always succeeds: every preference parks all cars and leaves one empty spot, denoted $E(A)\in\mathbb{Z}_m$.

\section{The closed formula}\label{sec:closed}

Let $\rho:\mathbb{Z}_m\to\mathbb{Z}_m$, $x\mapsto x\oplus 1$, act on preferences coordinatewise.

\begin{lemma}[Equivariance]\label{lem:equiv}
The rotation preserves admissibility, and $E(\rho A)=\rho(E(A))$ for every $A\in\mathcal{C}_m^{s,d}$.
\end{lemma}

\begin{proof}
Admissibility is defined by clockwise increments and the return gap, rotation invariants, so $\rho$ permutes $\mathcal{C}_m^{s,d}$. We show by induction on $j$ that after $j$ cars have parked, the occupied set for $\rho A$ is the $\rho$-image of that for $A$. It holds for $j=0$. Assuming it after $j-1$ cars, car $c_j$ in $\rho A$ tries $\rho(a_{1,j}),\dots$, and $\rho(a_{i,j})$ is occupied iff $a_{i,j}$ is occupied for $A$; the sweep from $\rho(a_{s,j})$ visits the $\rho$-images of the spots visited by the sweep from $a_{s,j}$, in order. Hence $c_j$ parks at the $\rho$-image of its spot in $A$; taking complements gives $E(\rho A)=\rho(E(A))$.
\end{proof}

\begin{theorem}\label{thm:closed}
Let $s\geq 1$, $d\geq 1$, $m\geq sd$ and $n=m-1$. For every spot $j\in\mathbb{Z}_m$,
$$\bigl|\{A\in\mathcal{C}_m^{s,d}:\ E(A)=j\}\bigr|=m^{\,n-1}\binom{m-sd+s-1}{s-1}^{\!n}.$$
\end{theorem}

\begin{proof}
By Lemma \ref{lem:equiv}, $\rho$ maps $\{E=j\}$ onto $\{E=\rho(j)\}$; these $m$ sets partition $\mathcal{C}_m^{s,d}$, so each has cardinality $|\mathcal{C}_m^{s,d}|/m$, which is the stated value by \eqref{eq:tuples}.
\end{proof}

\begin{corollary}[Two choices]\label{cor:s2}
For $s=2$ and $m=n+1$, $\ \bigl|\{E(A)=j\}\bigr|=(n+1)^{\,n-1}(n-2d+2)^{\,n}.$
\end{corollary}

\begin{remark}
For $s=2$, $d=1$ this is $(n+1)^{n-1}n^{n}$: the classical parking count appears as a factor. Section \ref{sec:refined} explains why. For $m<sd$ the model is empty, matching the vanishing binomial.
\end{remark}

\section{Refinement: equidistribution at fixed increments}\label{sec:refined}

The proof of Theorem \ref{thm:closed} in fact localizes to a much finer statement, which is the conceptual heart of the product structure. For a car $j$ let $\kappa_j:=(k_{1,j},\dots,k_{s-1,j})$ be its increment vector and $\kappa=(\kappa_1,\dots,\kappa_n)$ the increment matrix; the class
$$\mathcal{C}_\kappa:=\{A\in\mathcal{C}_m^{s,d}:\ A\text{ has increment matrix }\kappa\}$$
is parametrized by the anchor vector, so $|\mathcal{C}_\kappa|=m^{\,n}$, and there are $\binom{m-sd+s-1}{s-1}^{n}$ classes.

\begin{theorem}[Refined equidistribution]\label{thm:refined}
For every admissible increment matrix $\kappa$ and every spot $j\in\mathbb{Z}_m$,
$$\bigl|\{A\in\mathcal{C}_\kappa:\ E(A)=j\}\bigr|=m^{\,n-1}.$$
\end{theorem}

\begin{proof}
The rotation shifts all anchors by one and leaves the increments unchanged, so it maps $\mathcal{C}_\kappa$ bijectively onto itself; by Lemma \ref{lem:equiv} it maps $\{E=j\}\cap\mathcal{C}_\kappa$ onto $\{E=\rho(j)\}\cap\mathcal{C}_\kappa$. These $m$ sets partition $\mathcal{C}_\kappa$, so each has cardinality $m^{n}/m=m^{\,n-1}$.
\end{proof}

\begin{corollary}[Structural explanation]\label{cor:explain}
Summing over the $\binom{m-sd+s-1}{s-1}^{n}$ increment matrices recovers Theorem \ref{thm:closed}, exhibiting the count as a product with transparent factors:
$$\bigl|\{E=j\}\bigr|=\underbrace{\binom{m-sd+s-1}{s-1}^{\!n}}_{\text{free choice of increments}}\times\underbrace{m^{\,n-1}}_{\text{anchors, \`a la Pollak}}.$$
For $s=2$, $d=1$, the factorization $(n+1)^{n-1}n^{n}$ reads: $n^{n}$ free choices of one increment per car, times $(n+1)^{n-1}$ anchor configurations in every class --- the classical count arising because, at fixed increments, the anchors reproduce Pollak's situation.
\end{corollary}

\begin{remark}[Kaplansky's lemma and powers of cycles]
The unordered supports of the admissible tuples are the subsets of $\mathbb{Z}_m$ with all circular gaps at least $d$ --- the independent $s$-subsets of the $(d-1)$-th power of the cycle --- counted by Kaplansky's lemma \cite{Kaplansky} as $\frac{m}{s}\binom{m-sd+s-1}{s-1}$; each supports exactly $s$ admissible tuples, recovering \eqref{eq:tuples}. Independent subsets of powers of cycles are studied in \cite{Codara}. The model thus sits at the meeting point of Pollak's argument and Kaplansky's lemma.
\end{remark}

\begin{remark}[Numerical verification]\label{rem:verif}
Independently of the proofs, Theorems \ref{thm:closed} and \ref{thm:refined} were verified by exhaustive enumeration --- generating every preference, simulating the rule, and recording the empty-spot distribution overall and within each increment class. The equidistribution is exact and matches the formulas in all cases below.
\begin{center}
\begin{tabular}{c|r|r|r}
\hline
$(m,d,s)$ & admissible tuples & $|\{E=j\}|$ & classes ($\kappa$)\\
\hline
$(3,1,2)$ & $6$ & $12$ & $4$\\
$(4,1,2)$ & $12$ & $432$ & $27$\\
$(4,2,2)$ & $4$ & $16$ & $1$\\
$(5,1,2)$ & $20$ & $32\,000$ & $256$\\
$(5,2,2)$ & $10$ & $2\,000$ & $16$\\
$(4,1,3)$ & $12$ & $432$ & $27$\\
$(5,1,3)$ & $30$ & $162\,000$ & $1\,296$\\
$(6,2,3)$ & $6$ & $1\,296$ & $1$\\
$(5,1,4)$ & $20$ & $32\,000$ & $1$\\
\hline
\end{tabular}
\end{center}
In every increment class of every listed case, the empty spot is exactly equidistributed with value $m^{n-1}$.
\end{remark}

\section{Open problems}\label{sec:open}

\begin{problem}
Give a bijective proof of the factorization of Corollary \ref{cor:s2} for $d=1$, refining Theorem \ref{thm:refined} to a canonical bijection between each anchor class $\{E=j\}\cap\mathcal{C}_\kappa$ and the set of classical parking functions.
\end{problem}

\begin{problem}
Quantify the boundary effect between the circular and linear models. For $(n,d)=(3,1)$, $s=2$, the conditioned circular count is $432$, while exhaustive enumeration of the linear two-choice model gives $24$; characterize the circular preferences whose tuples and trajectories avoid crossing the empty spot, and relate them to the linear model.
\end{problem}

\begin{problem}
Study refined statistics on the circle (lucky cars, total clockwise displacement) and their generating functions at fixed increments, where Theorem \ref{thm:refined} suggests exact product answers.
\end{problem}

\paragraph*{Acknowledgements.}
The authors thank the RECITS Laboratory, USTHB, for its support.

\section*{Declarations}

\paragraph*{Funding.}
This research did not receive any specific grant from funding agencies in the public, commercial, or not-for-profit sectors.

\paragraph*{Competing interests.}
The authors have no competing interests to declare that are relevant to the content of this article.

\paragraph*{Data availability.}
The computer verification scripts supporting the results of this article are available from the corresponding author upon request.

\paragraph*{Use of artificial intelligence.}
All results are due to the authors, who reviewed and validated every statement and value; the AI assistant Claude (Anthropic) was used under their supervision for the numerical verification of the results by exhaustive enumeration, for assistance with the drafting and \LaTeX{} typesetting, and for linguistic proofreading, and the authors take full responsibility for the content of this article.

\end{document}